\documentclass[11pt]{amsart}
\usepackage[utf8]{inputenc}
\usepackage{amsfonts, amssymb, amsmath, amsthm, color, float, enumerate}
\usepackage{url}
\usepackage{mathrsfs}
\usepackage[unicode,psdextra]{hyperref}
\usepackage{pgfplots,tikz}
\pgfplotsset{compat=1.18}
\usepackage{todonotes}

\title[Maximal operators associated to critical radius functions]{Characterization of the continuity properties of maximal operators associated to critical radius functions via Dini type conditions}
\author{}
\date{}

\usepackage[a4paper, left=2cm, right=2cm, top=3cm, bottom=3cm]{geometry} 

\theoremstyle{plain}
   \newtheorem{teo}{Theorem}
   \newtheorem{coro}[teo]{Corollary}
   \newtheorem{lema}[teo]{Lemma}
   \newtheorem{propo}[teo]{Proposition}
   
\theoremstyle{definition}
   
\theoremstyle{remark}

 \newtheorem{afirmacion}{Claim}

\numberwithin{equation}{section}
\numberwithin{teo}{section}
\allowdisplaybreaks

\definecolor{aquamarine}{rgb}{0.5, 1.0, 0.83}
\definecolor{americanrose}{rgb}{1.0, 0.01, 0.24}
\definecolor{arsenic}{rgb}{0.23, 0.27, 0.29}
\definecolor{blizzardblue}{rgb}{0.67, 0.9, 0.93}
\definecolor{blush}{rgb}{0.87, 0.36, 0.51}
\definecolor{celestialblue}{rgb}{0.29, 0.59, 0.82}
\definecolor{chocolate(web)}{rgb}{0.82, 0.41, 0.12}
\definecolor{brightpink}{rgb}{1,0,0.5}
\definecolor{cadmiunred}{rgb}{0.89,0,0.13}

\hypersetup{
	colorlinks = true,
	linkcolor = cadmiunred,
	anchorcolor = blue,
	citecolor = brightpink,
	filecolor = blue,
	urlcolor = blue
}

\newcounter{BP}

\begin{document}

\begin{abstract}
We give a characterization of the continuity properties of a Luxemburg maximal type operator associated to a critical radius function $\rho$ between Orlicz spaces. This goal is achieved by means of a Dini type condition that includes certain Young functions related to the maximal operator and the spaces involved. Our results provide not only weak Fefferman-Stein type inequalities but also a weak weighted estimate of modular type for the considered operators, which is interesting in its own right. 
On the other hand, we prove the boundedness of the Hardy-Littlewood maximal function associated to $\rho$ between Zygmund spaces of $L\,\log\,L$ type with $A_p$ weights.

\end{abstract}

	\author[F. Berra]{Fabio Berra}
	\address{Fabio Berra, CONICET and Departamento de Matem\'{a}tica (FIQ-UNL),  Santa Fe, Argentina.}
	\email{fberra@santafe-conicet.gov.ar}
	
	\author[M. Carena]{Marilina Carena}
	\address{Marilina Carena, CONICET and FICH (UNL),  Santa Fe, Argentina.}
	\email{marilcarena@gmail.com}

	\author[G. Pradolini]{Gladis Pradolini}
	\address{Gladis Pradolini, CONICET and Departamento de Matem\'{a}tica (FIQ-UNL),  Santa Fe, Argentina.}
	\email{gladis.pradolini@gmail.com}
	
	\thanks{The authors were supported by CONICET, UNL and Gobierno de la Provincia de Santa Fe}
	
	

\maketitle

\section{Introduction}

It is well known that many continuity properties of certain operators in Harmonic Analysis can be directly obtained by studying the corresponding problem for an adequate maximal operator. This behavior is due to a certain control that the latter exerts over the former, that is usually given in the norm of the spaces  where the operators act. A typical inequality that illustrates this fact is given by
\begin{equation}\label{eq: intro - desigualdad de Coifman}
  \int_{\mathbb{R}^n}|\mathcal{T}f(x)|^pw(x)\,dx\leq C\int_{\mathbb{R}^n}\mathcal{M}_{\mathcal{T}}f(x)^pw(x)\,dx  
\end{equation}
where $w$ is a weight, $\mathcal{T}$ is a certain integral operator and $\mathcal{M}_{\mathcal{T}}$ is a maximal operator associated to $\mathcal{T}$. This estimate shows that $\mathcal{T}$ inherits the continuity properties of $\mathcal{M}_{\mathcal{T}}$ in the weighted Lebesgue spaces $L^p(w)$. Several examples are in order for the pair $(\mathcal{T},\mathcal{M}_\mathcal T)$: in  \cite{C72} and \cite{CF74} the authors established~\eqref{eq: intro - desigualdad de Coifman} for the unweighted case and $A_\infty$ weights  for $(T,M)$, respectively, where $T$ is a Calderón-Zygmund operator (CZO) and $M$ the classical Hardy-Littlewood maximal function. For the pair $(T_b^m, M^{m+1})$ this result was proved in \cite{Perez_Sharp97}, where $T_b^m$ is a commutator of order $m$ of $T$ with BMO symbol and $M^{m+1}$ is the $m+1\,$-\,iteration of $M$. It is well known that these iterations of $M$ are equivalent to certain maximal operators associated to a Young function of $L\,$log$\,L$ type (see, for example \cite{B-L-P-R}). 

Regarding operators of convolution type with kernel satisfying certain generalized Hörmander conditions, the corresponding maximal operators are defined by means of more general Young functions (see \cite{LMRdlT} and \cite{Lorente-Riveros-delaTorre05}). 

 Inequalities of Fefferman-Stein type for the operators mentioned above were also considered by many authors (\cite{FS71}, \cite{Kanashiro-Pradolini}, \cite{LMRdlT}, \cite{Perez94}, \cite{PP2001}, \cite{Wilson-89}). It is appropriate to point out that no condition on the weights is assumed in these estimates. 

  The previous discussion shows the relevance of studying  continuity properties of maximal type operators. In this direction, in \cite{Perez-95-Onsuf} Pérez proved an important characterization of the boundedness between Lebesgue spaces of Luxemburg maximal operators associated to a Young function $\Phi$ belonging to certain class related to the underlying spaces. This characterization also includes Fefferman-Stein type inequalities. In a more general setting, in \cite{PW01}, the authors extend this characterization to spaces of homogeneous type under the assumption that every annuli is nonempty, condition that was then removed in \cite{PS04}. Later on, in \cite{Kanashiro-Pradolini}, an extension involving Orlicz spaces and Dini type conditions was given. 

In the Schrödinger setting, in \cite{BCH13} the authors study inequalities in the spirit of \eqref{eq: intro - desigualdad de Coifman} for operators associated to a critical radius function, that is, a function $\rho\colon \mathbb{R}^n\to (0,\infty)$ whose variation is controlled by the existence of two constants $C_0,N_0\geq 1$ such that the inequality
	\begin{equation} \label{eq: constantesRho}
	C_0^{-1}\rho(x) \left(1+ \frac{|x-y|}{\rho(x)}\right)^{-N_0}
	\leq \rho(y)
	\leq C_0 \,\rho(x) \left(1+ \frac{|x-y|}{\rho(x)}\right)^{\tfrac{N_0}{N_0+1}}
	\end{equation}
	holds for every $x,y\in\mathbb{R}^n$. We shall be dealing with the Euclidean space $\mathbb{R}^n$ equipped with this function.

 In this paper we shall be concerned with a problem of the type described above by considering maximal operators of Luxemburg type, $M_\eta^{\rho,\sigma}$, where $\rho$ is a critical radius function satisfying \eqref{eq: constantesRho}. 

 We now introduce the definitions of the maximal operators that will be considered in this article. For further details see also Section~\ref{section: preliminares}.  Given a locally integrable function $f$ and $\sigma\geq 0$, the \textit{Hardy-Littlewood maximal operator associated to~$\rho$}  is defined by
\begin{equation}\label{eq: operador maximal de H-L}
M^{\rho,\sigma}f(x)=\sup_{Q(x_0,r_0)\ni x} \left(1+\frac{r_0}{\rho(x_0)}\right)^{-\sigma}\left(\frac{1}{|Q|}\int_Q |f(y)|\,dy\right),
\end{equation}
where $Q(x_0,r_0)$ stands for the cube with sides parallel to the coordinate axes centered at $x_0$ and with radius~$r_0$, that is, $r_0=\sqrt{n}\,\ell(Q)/2$.

Let $\eta$ be a Young function. For $\sigma\geq 0$ we define
\begin{equation}\label{eq: Maximal generalizada}
  M_\eta^{\rho,\sigma}f(x)=\sup_{Q(x_0,r_0)\ni x} \left(1+\frac{r_0}{\rho(x_0)}\right)^{-\sigma}\|f\|_{\eta,Q},  
\end{equation}
where $\|f\|_{\eta,Q}$ denotes the Luxemburg average of $f$ over the cube $Q$, given by
\[\|f\|_{\eta,Q}=\inf\left\{\lambda>0: \frac{1}{|Q|}\int_Q \eta\left(\frac{|f(x)|}{\lambda}\right)\,dx\leq 1\right\}.\]
When $\sigma=0$ we have that $M_\eta^{\rho,\sigma}=M_\eta$, the classical generalized maximal function associated to $\eta$. 

Given a weight $w$ and a Young function $\Phi$, the \emph{Zygmund space} $L^{\Phi}(w)$ is defined by
\[L^{\Phi}(w)=\left\{f \text{ measurable }: \varrho_{\Phi,w}(f/\lambda)<\infty \text{ for some 
}\lambda>0\right\},\]
where 
\[\varrho_{\Phi,w}(f/\lambda)=\int_{\mathbb{R}^n}\Phi\left(\frac{|f(x)|}{\lambda}\right)w(x)\,dx.\]

For $f\in L^\Phi(w)$, the functional
\[\|f\|_{\Phi,w}=\inf\left\{\lambda>0: \varrho_{\Phi,w}(f/\lambda)\leq 1\right\}\]
is a norm on this space. Moreover, $(L^\Phi(w), \|\cdot\|_{\Phi,w})$ is a Banach space. When $w=1$, we directly write $\varrho_{\Phi,w}=\varrho_\Phi$ and $L^\Phi(w)=L^\Phi$.

Let $a$ and $b$ be positive continuous functions defined on $[0,\infty)$ such that $a(0)=b(0)=0$, and we shall also assume that $b$ is non decreasing and satisfies $\displaystyle \lim_{s\to \infty} b(s)=\infty$.
Let us also consider the functions $\phi$ and $\psi$ given by
\begin{equation}\label{eq: definicion de phi y psi}
  \phi(t)=\int_0^t a(s)\,ds \quad\textup{ and }\quad \psi(t)=\int_0^t b(s)\,ds.  
\end{equation}
 Under the assumptions on $b$, $\psi$ is a Young function.

We are now in position to state the main results of this work. The first theorem contains an important characterization of the continuity properties of $M_\eta^{\rho,\sigma}$ by means of a Dini type condition that involves $\eta$ and the corresponding functions associated to the underlying spaces. This condition generalizes the well known $B_p$ condition introduced by Pérez in \cite{Perez-95-Onsuf}. This result requires a weighted modular Fefferman-Stein type estimate, leading to a modular weak type inequality for $M_\eta^{\rho,\sigma}$ which is interesting by itself.

\begin{teo}\label{teo: equivalencias}
Let $\eta$ be a normalized and differentiable Young function in $\Delta_2$. Let $a,b,\phi$ and $\psi$ functions defined as in~\eqref{eq: definicion de phi y psi}. Then the following statements are equivalent:
\begin{enumerate}[\rm(a)]
    \item\label{item: teo: equivalencias - item a} There exists a positive constant $C$ such that the inequality
    \[\int_0^{t}\frac{a(s)}{s}\eta'(t/s)\,ds\leq Cb(Ct)\]
    holds for every $t\geq 0$.
    \item\label{item: teo: equivalencias - item b} Given any $\theta\geq 0$, there exist nonnegative constants $C$ and $\sigma$ such that the inequality
    \[\int_{\mathbb{R}^n}\phi(M_\eta^{\rho,\sigma} f(x))w(x)\,dx\leq C\int_{\mathbb{R}^n} \psi(C|f(x)|)M^{\rho,\theta}w(x)\,dx\]
    holds for every $w\geq 0$ and every function $f\in L^{\psi}(M^{\rho,\theta}w)$.
    \item\label{item: teo: equivalencias - item c} For every $\theta\geq 0$, there exist nonnegative constants $C$ and $\sigma$  such that the estimate
    \[\|M_\eta^{\rho,\sigma}f\|_{\phi,w}\leq C\|f\|_{\psi, M^{\rho,\theta}w}\]
    holds for every $w\geq 0$.
    \item \label{item: teo: equivalencias - item d} There exist nonnegative constants $C$ and $\sigma$  such that the inequality
    \[\int_{\mathbb{R}^n}\phi(M_\eta^{\rho,\sigma} f(x))\,dx\leq C\int_{\mathbb{R}^n} \psi(C|f(x)|)\,dx\]
    holds for every $f\in L^{\psi}$.
    \item  \label{item: teo: equivalencias - item e} For every $\theta\geq 0$ there exists $C, \sigma\geq 0$ such that the inequality
    \[\int_{\mathbb{R}^n}\phi\left(\frac{M^{\rho,\gamma}f(x)}{M_{\tilde \eta}^{\rho,\gamma-\sigma}u(x)}\right)w(x)\,dx\leq C\int_{\mathbb{R}^n} \psi\left(\frac{|f(x)|}{u(x)}\right)M^{\rho,\theta}w(x)\,dx\]
    holds for every $\gamma\geq \sigma$ and every nonnegative functions $f$, $u$ and $w$.
\end{enumerate}
\end{teo}

We point out that if we further assume that $\phi$ is a Young function, we can add another condition to the characterization given above.

\begin{teo}\label{teo: equivalencias ampliado}
Let $\eta,a,b,\phi$ and $\psi$ functions satisfying the same conditions as in Theorem~\ref{teo: equivalencias}. If $\phi$ is a Young function, we have that items~\eqref{item: teo: equivalencias - item a} through~\eqref{item: teo: equivalencias - item e} are equivalent to the following condition.
\begin{enumerate}[\rm(a)]
\setcounter{enumi}{+5}
  \item \label{item: teo: equivalencias - item f} There exist  constants $C>0$ and $\sigma\geq 0$ such that
     \[\|M_\eta^{\rho,\sigma}f\|_{\phi}\leq C\|f\|_{\psi}\]
     holds for every $f\in L^{\psi}$.
\end{enumerate}
\end{teo}

For the classical versions of the maximal operators involved, the theorems above were obtained in \cite{Perez-95-Onsuf} when $a(t)=b(t)=t^{p-1}$, $p>1$, and in \cite{Kanashiro-Pradolini} for the Euclidean case.

The next result establishes a sufficient condition in order to guarantee the continuity of $M^{\rho,\theta}$ between  weighted Zygmund spaces of $L\log L$ type (for the definition of the classes of weights see Section~\ref{section: preliminares}).

\begin{teo}\label{teo: acotacion de M en LpLogLq}
    Let $p>1$, $q\geq 0$ and $\Phi_{p,q}(t)=t^p(1+\log^+t)^q$. If $w\in A_p^{\rho}$, then there exists $\theta\geq 0$ such that $M^{\rho,\theta}$ is bounded on $L^{\Phi_{p,q}}(w)$ .
\end{teo}

The following theorem contains a Fefferman-Stein weak type inequality for the operator $M_\eta^{\rho,\sigma}$.

\begin{teo}\label{teo: tipo debil modular de MPhi,rho}
    Let $w$ be a weight and $\Phi$ a Young function in $\Delta_2$. For every $\theta\geq 0$, there exist $\sigma\geq 0$ and $C>0$ such that the inequality
    \[w\left(\left\{x\in\mathbb{R}^n: M_{\Phi}^{\rho,\sigma} f(x)>\lambda\right\}\right)\leq C\int_{\mathbb{R}^n}\Phi\left(\frac{|f(x)|}{\lambda}\right)M^{\rho,\theta}w(x)\,dx\]
    holds for every positive $\lambda$.
\end{teo}

Although this estimate is interesting by itself, it will play an important role
in the proof of Theorem~\ref{teo: equivalencias}.

\begin{coro}
If  $w\in A_1^{\rho}$, there exist $C>0$ and $\sigma\geq 0$ such that
    \[w\left(\left\{x\in\mathbb{R}^n: M_{\Phi}^{\rho,\sigma} f(x)>\lambda\right\}\right)\leq C\int_{\mathbb{R}^n}\Phi\left(\frac{|f(x)|}{\lambda}\right)w(x)\,dx.\]
\end{coro}

The article is organized as follows. In Section~\ref{section: preliminares} we give the preliminaries and definitions.  Section~\ref{section: prueba cond suficiente} and~\ref{seccion: prueba debiles} will be devoted to the proof of Theorem~\ref{teo: acotacion de M en LpLogLq} and Theorem~\ref{teo: tipo debil modular de MPhi,rho}, respectively. We postpone the proof of the main results, Theorem~\ref{teo: equivalencias} and Theorem~\ref{teo: equivalencias ampliado}, to Section~\ref{section: prueba equivalencias} because they require the estimate given by Theorem~\ref{teo: tipo debil modular de MPhi,rho}.

\section{Preliminaries and definitions}~\label{section: preliminares}

Throughout the article we shall assume that $\rho\colon \mathbb{R}^n\to (0,\infty)$ is a fixed critical radius function satisfying \eqref{eq: constantesRho}. 

Let us introduce the classes of weights involved in our estimates. These types of Muckenhoupt $A_p$ classes related to $\rho$ were first defined in~\cite{BHS-classesofweights}.
Let $1<p<\infty$ and $\theta\geq 0$. We say that $w\in A_p^{\rho,\theta}$ if there exists a positive constant $C$ such that the inequality
\begin{equation}\label{eq: clase Ap,rho,theta}
\left(\frac{1}{|Q|}\int_Q w\right)^{1/p}\left(\frac{1}{|Q|}\int_Q w^{1-p'}\right)^{1/p'}\leq C\left(1+\frac{r}{\rho(x)}\right)^{\theta}
\end{equation} 
holds for every cube $Q=Q(x,r)$. The notation $Q(x,r)$ stands for a cube with sides parallel to the coordinate axes, with center $x$ and radius $r$. We shall denote $\ell(Q)$ the length of the edges of $Q$. Observe that $r=\sqrt{n}\,\ell(Q)/2$.

Similarly, $w\in A_1^{\rho,\theta}$ if there exists $C>0$ such that
\begin{equation}\label{eq: clase A1,rho,theta}
\frac{1}{|Q|}\int_Q w\leq C\left(1+\frac{r}{\rho(x)}\right)^{\theta}\inf_Q w,
\end{equation} 
for every cube $Q=Q(x,r)$.  It is well known that $A_1^{\rho,\theta}$ weights verify that $M^{\rho,\theta} w(x)\leq Cw(x)$ for almost every $x$ (see, for example, \cite{BCH12}). We also define $A_\infty^{\rho,\theta}=\bigcup_{p\geq 1} A_p^{\rho,\theta}$.

For $1\leq p\leq \infty$, the $A_p^\rho$ class is defined as the collection of all the $A_p^{\rho,\theta}$ classes for $\theta\geq 0$, that is
\[A_p^\rho=\bigcup_{\theta\geq 0} A_p^{\rho,\theta}.\]

We shall denote with $\mathcal{Q}_\rho$ the family of cubes $Q=Q(x,r)$ such that $r\leq \rho(x)$. When $r<\rho(x)$ or $r=\rho(x)$ we shall say that $Q$ is a \emph{subcritical} or a \emph{critical} cube, respectively.
Note that, if $Q\in \mathcal{Q}_\rho$, the expression $\left(1+\frac{r}{\rho(x)}\right)^\theta$ is equivalent to a constant.

We say that $\Phi\colon [0,\infty)\to [0,\infty)$ is a \emph{Young function} if it is increasing, convex, $\Phi(0)=0$ and $\Phi(t)\to\infty$ when $t\to \infty$. We shall refer to  $\Phi$ as a \emph{normalized Young function} if $\Phi(1)=1$. It is easy to check that any Young function can be normalized. We also say that $\Phi$ is \emph{doubling}, and denote it by $\Phi\in \Delta_2$, if there exists a positive constant $C$ such that 
\[\Phi(2t)\leq C\Phi(t),\]
for every $t\geq 0$.

Given a Young function $\Phi$, the \emph{complementary function} of $\Phi$ is defined by
\[\tilde\Phi(t)=\sup\{st-\Phi(s): s\geq 0\}.\]
Moreover, if $\Phi(t)/t$ is a quasi increasing function, then $\tilde \Phi$ is also a Young function. The following relation between $\Phi$ and $\tilde\Phi$ 
\begin{equation}\label{eq: relacion Phi - tilde Phi}
    \frac{t}{2}\leq \Phi^{-1}(t)\tilde\Phi^{-1}(t)\leq 2t
\end{equation}
holds for every $t>0$. The inequality above implies the 
generalized Hölder inequality
\[\int_{\mathbb{R}^n} |fg|\leq C \|f\|_{\Phi}\|g\|_{\tilde\Phi}.\]
Furthermore, for every cube $Q$ we can derive the inequality
\[\frac{1}{|Q|}\int_{Q} |fg|\leq C \|f\|_{\Phi,Q}\|g\|_{\tilde\Phi,Q}.\]

We shall be using the following equivalence in our estimations  
\begin{equation}\label{eq: relacion norma Luxemburgo con infimo}
    \|f\|_{\eta,Q}\approx \inf_{t>0}\left\{t+\frac{t}{|Q|}\int_Q\eta\left(\frac{|f|}{t}\right)\right\},
\end{equation}
that relates the Luxemburg average of $f$ with its modular version. For more information about Orlicz spaces see, for example, \cite{KR} or \cite{raoren}.

\section{Proof of Theorem~\ref{teo: acotacion de M en LpLogLq}}\label{section: prueba cond suficiente}

We shall require the openness property for $A_p^\rho$ weights, given in the following lemma. A proof can be found in \cite{BHS-classesofweights}.

\begin{lema}\label{lema: apertura clases Ap,rho}
    Let $p>1$ and $w\in A_p^\rho$. Then there exists $0<\varepsilon<p-1$ such that $w\in A_{p-\varepsilon}^\rho$.
\end{lema}

The following result establishes the relation between the boundedness of $M^{\rho,\theta}$ on $L^p(w)$ and the $A_p^\rho$ classes (see Proposition 3 in \cite{BCHextrapolation}).

\begin{propo}\label{propo: tipo fuerte de M^rho, theta con pesos Ap}
    Let $1<p<\infty$. A weight $w$ belongs to $A_p^\rho$ if and only if there exists $\theta\geq 0$ such that $M^{\rho,\theta}$ is bounded on $L^p(w)$.
\end{propo}

We shall now proceed to the proof of Theorem~\ref{teo: acotacion de M en LpLogLq}. 

\begin{proof}[Proof of Theorem~\ref{teo: acotacion de M en LpLogLq}]
 Let $w\in A_p^\rho$ and $f\in L^{\Phi_{p,q}}(w)$, where $\Phi_{p,q}(t)=t^p(1+\log^+ t)^q$. We can assume, without loss of generality, that $f$ is nonnegative and $\|f\|_{\Phi_{p,q},w}=1$. By Lemma~\ref{lema: apertura clases Ap,rho}, there exists $0<\varepsilon<p-1$ such that $w\in A_{p-\varepsilon}^\rho$ and consequently, by Proposition~\ref{propo: tipo fuerte de M^rho, theta con pesos Ap}, there exists $\sigma\geq 0$ such that  $M^{\rho,\sigma}$ is bounded on $L^{p-\varepsilon}(w)$. Let $a=p-\varepsilon>1$.  We shall first prove that
\begin{equation}\label{eq: teo: acotacion de M en LpLogLq - eq1}
\Phi_{p,q}\left(M^{\rho,\theta}f(x)\right)\leq \left[M^{\rho,\sigma}\left({\Phi}_{p/a,q/a}(f)\right)(x)\right]^a,   
\end{equation}
where $\theta=\sigma a$,

Indeed, fix $x$ and $Q=Q(x_0,r_0)$ a cube containing $x$. Since both $\Phi_{p,q}$ and $ \Phi_{p/a,q/a}$ are convex functions, we have
\begin{align*}
    \Phi_{p,q}\left(\left(1+\frac{r_0}{\rho(x_0)}\right)^{-\theta}\frac{1}{|Q|}\int_Q f\right)&\leq \left(1+\frac{r_0}{\rho(x_0)}\right)^{-a\sigma}\Phi_{p,q}\left(\frac{1}{|Q|}\int_Q f\right)\\
    &= \left(1+\frac{r_0}{\rho(x_0)}\right)^{-a\sigma}\left[{\Phi}_{p/a,q/a}\left(\frac{1}{|Q|}\int_Q f\right)\right]^a\\
    &\leq \left[\left(1+\frac{r_0}{\rho(x_0)}\right)^{-\sigma}\frac{1}{|Q|}\int_Q {\Phi}_{p/a,q/a}(f)\right]^a\\
    &\leq \left[M^{\rho,\sigma}\left({\Phi}_{p/a,q/a}(f)\right)(x)\right]^a.
\end{align*}
By taking supremum over the cubes $Q$, and using the fact that $\Phi_{p,q}$ is continuous and increasing, we get the estimate \eqref{eq: teo: acotacion de M en LpLogLq - eq1}. 

We now proceed as follows
\begin{align*}
\int_{\mathbb{R}^n}\Phi_{p,q}\left(M^{\rho,\theta}f(x)\right)w(x)\,dx&\leq \int_{\mathbb{R}^n} \left[M^{\rho,\sigma}\left({\Phi}_{p/a,q/a}(f)\right)(x)\right]^aw(x)\,dx\\
&\leq C\int_{\mathbb{R}^n} \left({\Phi}_{p/a,q/a}(f(x))\right)^a w(x)\,dx\\
&=C\int_{\mathbb{R}^n} \Phi_{p,q}(f(x)) w(x)\,dx\\
&= C,
\end{align*}
where we have used that $w\in A_a^\rho$. By a simple computation we have the result.  
\end{proof}

\section{Proof of Theorem~\ref{teo: tipo debil modular de MPhi,rho}}\label{seccion: prueba debiles}

We start by giving some previous results in order to prove Theorem~\ref{teo: tipo debil modular de MPhi,rho}.
The following proposition gives a partition of $\mathbb{R}^n$ in terms of critical cubes. A proof can be found in \cite{DZ-99}.
 \begin{propo}
\label{propo: cubrimiento critico}
	There exists a sequence of points $\{x_j\}_{j\in\mathbb{N}}$ such that the family of critical cubes given by $Q_j=~Q(x_j,\rho(x_j))$  satisfies
	\begin{enumerate}[\rm(a)]
		\item \label{item: prop-cubrimientocritico - item a}$\displaystyle \bigcup_{j\in\mathbb{N}} Q_j= \mathbb{R}^n$.
		\item \label{item: prop-cubrimientocritico - item b}There exist positive constants $C$ and $N_1$ such that for any $\sigma\geq1$,
		$\displaystyle \sum_{j\in\mathbb{N}}\mathcal{X}_{\sigma Q_j}\leq C \sigma^{N_1}$. Here $\mathcal{X}_A$ denotes, as usual, the characteristic function of the set $A$.
	\end{enumerate}
 \end{propo}

We shall be dealing with two versions of the operator $M_\Phi$. Concretely, given a fixed cube $R$ we define
\[M_{\Phi,R}f(x)=\sup_{Q\ni x, Q\subseteq R}\|f\|_{\Phi, Q},\]
and
\[M_{\Phi,R}^{\mathscr{D}}f(x)=\sup_{Q\ni x, Q\in \mathcal{D}(R)}\|f\|_{\Phi, Q},\]
where $\mathcal{D}(R)$ denotes the family of cubes obtained from $R$ by dividing it dyadically.

By means of Proposition~\ref{propo: cubrimiento critico}, it will be enough to achieve the estimate claimed in Theorem~\ref{teo: tipo debil modular de MPhi,rho} for cubes that are a constant multiple of a critical cube. The following proposition contains this estimate for the auxiliary operator $M_{\Phi,R}^{\mathscr{D}}$.

\begin{propo}\label{propo: tipo debil modular en un cubo}
    Let $\Phi$ be a Young function, $w$ a weight and $R=Q(x_R, C\,\rho(x_R))$ be a cube, where $C>0$. Then, for every $\theta\geq 0$ there exists a positive constant $C_{n,\rho,\theta}$ such that the inequality
    \[w(\{x\in R: M_{\Phi,R}^{\mathscr{D}}f(x)>\lambda\})\leq C_{n,\rho,\theta}\int_R \Phi\left(\frac{|f(x)|}{\lambda}\right)M^{\rho,\theta}w(x)\,dx\]
    holds for every positive $\lambda$ and every bounded function $f$ with compact support.
\end{propo}

\begin{proof}
    We first consider the case $\|f\|_{\Phi,R}>\lambda$. This is equivalent to 
    \[1<\frac{1}{|R|}\int_R \Phi\left(\frac{|f|}{\lambda}\right).\]
    Then the estimate easily holds as follows
    \[w(\{x\in R: M_{\Phi,R}^{\mathscr{D}}f(x)>\lambda\})\leq w(R)\leq \frac{w(R)}{|R|}\int_R \Phi\left(\frac{|f|}{\lambda}\right)\leq (1+C)^\theta\int_R\Phi\left(\frac{|f|}{\lambda}\right)M^{\rho,\theta}w.\]
    Now assume that $\|f\|_{\Phi, R}\leq \lambda$. For $k\in\mathbb{N}$, we define
    \[\Lambda_{k}(R)=\{Q\in \mathcal{D}(R): \ell(Q)=2^{-k}\ell(R)\}.\]
    We pick those cubes $Q$ in $\Lambda_1$ such that
    \[\|f\|_{\Phi,Q}>\lambda\]
    and we bisect the others. Repeating this process indefinitely, we obtain a sequence of disjoint cubes $\{Q_j\}_j$ in $\mathcal{D}(R)$ verifying that $\|f\|_{\Phi,Q_j}>\lambda$ and also
    \[\{x\in R: M_{\Phi,R}^{\mathscr{D}}f(x)>\lambda\}=\bigcup_j Q_j.\]
    If we write $Q_j=Q(x_j, r_j)$ for every $j$, by using \eqref{eq: constantesRho} we have that
    \begin{align*}
        \left(1+\frac{r_j}{\rho(x_j)}\right)^\theta&\leq \left(1+C\frac{\rho(x_R)}{\rho(x_j)}\right)^\theta\\
        &\leq \left(1+CC_0\left(1+\frac{|x_R-x_j|}{\rho(x_R)}\right)^{N_0}\right)^\theta\\
        &\leq \left(1+CC_0\left(1+C\right)^{N_0}\right)^\theta.
    \end{align*}
    If we take $C_{n,\rho,\theta}=\left(1+CC_0\left(1+C\right)^{N_0}\right)^\theta$, we can proceed as follows
    \begin{align*}
   w\left(\left\{x\in R: M_{\Phi,R}^{\mathscr{D}}f(x)>\lambda\right\}\right)=\sum_jw(Q_j)&\leq \sum_j \frac{w(Q_j)}{|Q_j|}\int_{Q_j}\Phi\left(\frac{|f|}{\lambda}\right)\\
   &\leq C_{n,\rho,\theta} \sum_j \int_{Q_j}\Phi\left(\frac{|f|}{\lambda}\right)M^{\rho,\theta}w\\
    &\leq C_{n,\rho,\theta}\int_R\Phi\left(\frac{|f|}{\lambda}\right)M^{\rho,\theta}w.     
    \end{align*}
    Since $(1+C)^\theta\leq C_{n,\rho,\theta}$, we can conclude the desired estimate.
\end{proof}

Before stating the next lemma, we recall a useful definition for the sequel.
A \emph{dyadic grid} $\mathcal{D}$ is understood as a collection of cubes in $\mathbb{R}^n$ with the following properties:
\begin{enumerate}
	\item every cube  $Q$ in $\mathcal{D}$ verifies $\ell(Q)=2^k$, for some $k\in\mathbb{Z}$;
	\item if $P$ and $Q$ are in $\mathcal{D}$ and $P\cap Q\neq\emptyset$, then either $P\subseteq Q$ or $Q\subseteq P$;
	\item $\mathcal{D}_k=\{Q\in \mathcal{D}: \ell(Q)=2^k\}$ is a partition of $\mathbb{R}^n$, for every $k\in \mathbb{Z}$.
\end{enumerate}

The following lemma establishes an important geometric relation between cubes in $\mathbb{R}^n$ and dyadic grids (see \cite{Okikiolu}).

\begin{lema}\label{lema: cubrimiento por grillas diadicas}
	For each $1\leq i\leq 3^n$, there exist dyadic grids $\mathcal{D}^{(i)}$, such that for every cube $Q$ in $\mathbb{R}^n$ there exist an index $1\leq i_0\leq 3^n$ and a dyadic cube $Q_0\in \mathcal{D}^{(i_0)}$ with $Q\subseteq Q_0$ and $\ell(Q_0)\leq 3\ell(Q)$. 
\end{lema} 

The next lemma provides a relation between the auxiliary operators  $M_{\Phi, R}$ and $M_{\Phi, R}^{\mathscr{D}}$. Although a version for the case $\Phi(t)=t$ was proved in \cite{BPQ}, we include the details for the general case for the sake of completeness.  

\begin{lema}\label{lema: MPhi local a MPhi local diadica}
	For each $1\leq i\leq 3^n$ there exist dyadic grids $\mathcal{D}^{(i)}$  with the following property:
	for every cube $Q$ in $\mathbb{R}^n$ there exists $3^n$ dyadic cubes $Q_i\in \mathcal{D}^{(i)}$  such that
	\[M_{\Phi, Q} f(x)\leq 3^n \sum_{i=1}^{3^n}M_{\Phi, Q_i}^{\mathscr{D}}(f\mathcal{X}_{Q})(x),\]
	for every $x\in Q$. Furthermore, each $Q_i$ verifies $Q\subseteq Q_i\subseteq C Q$, where $C$ depends only on $n$.
\end{lema}

\begin{proof}
    Given a cube $Q_0$ and a dyadic grid $\mathcal{D}$,  let us denote by $M_{\Phi, Q_0,\mathcal{D}}f$ the version of $M_{\Phi, Q_0}$ where the supremum is taken only for cubes in $\mathcal{D}$ contained in $Q_0$, that is
	\[M_{\Phi, Q_0,\mathcal{D}}f(x)=\sup_{Q\in \mathcal{D}, Q\subseteq Q_0}\|f\|_{\Phi,Q}.\]
	
	Fix $x\in Q$ and let $P\subseteq Q$ be a subcube containing $x$. Fixed $1\leq i\leq 3^n$, by Lemma~\ref{lema: cubrimiento por grillas diadicas} there exist a dyadic grid $\mathcal{D}^{(i)}$ and $P_i\in\mathcal{D}^{(i)}$ such that $P\subseteq P_i$ and $\ell(P_i)\leq 3\ell(P)$. We claim that $P_i\subseteq 8\sqrt{n}Q$. Indeed, if $x_Q$ denotes the center of $Q$, for $y\in P_i$ we have that
	\[|y-x_Q|\leq |y-x|+|x-x_Q|\leq \sqrt{n}\ell(P_i)+\frac{\sqrt{n}}{2}\ell(Q)< 4\sqrt{n}\ell(Q),\]
	so $P_i\subseteq B(x_Q, 4\sqrt{n}\ell(Q))\subseteq 8\sqrt{n}Q$.
	Since
 \[\frac{1}{|P|}\int_P \Phi\left(\frac{|f|\mathcal{X}_Q}{\|f\|_{\Phi, P_i}}\right)\leq \frac{3^n}{|P_i|}\int_{P_i}\Phi\left(\frac{|f|\mathcal{X}_Q}{\|f\|_{\Phi, P_i}}\right)\leq 3^n,\]
 we have that $\|f\mathcal{X}_Q\|_{\Phi,P}\leq 3^n\|f\mathcal{X}_Q\|_{\Phi,P_i}$. This yields
 \[\|f\mathcal{X}_Q\|_{\Phi,P}\leq 3^n\|f\mathcal{X}_Q\|_{\Phi,P_i}\leq 3^n M_{\Phi, 8\sqrt{n}Q, \mathcal{D}^{(i)}}(f\mathcal{X}_Q)(x)\leq 3^n\sum_{i=1}^{3^n}M_{\Phi, 8\sqrt{n}Q, \mathcal{D}^{(i)}}(f\mathcal{X}_Q)(x).\]
	
By taking supremum over the cubes $P\subseteq Q$ we arrive to
\[M_{\Phi, Q}f(x)\leq 3^n\sum_{i=1}^{3^n}M_{\Phi, 8\sqrt{n}Q,\mathcal{D}^{(i)}}(f\mathcal{X}_Q)(x),\]
for every $x\in Q$.

 Fix the unique $k\in\mathbb{Z}$ such that
    \begin{equation}\label{eq: lema: MPhi local a MPhi local diadica - eq1}
     2^k<8\sqrt{n}\ell(Q)\leq 2^{k+1}.   
    \end{equation}
    There exist at most $2^n$ cubes in $\mathcal{D}^{(i)}$ with side length $2^k$ and that intersect $8\sqrt{n}Q$. Let $Q_i$ be the smallest dyadic cube in  $\mathcal{D}^{(i)}$ that contains these cubes, which implies that $\ell(Q_i)=2^{k+1}$. We claim that
    \[8\sqrt{n}Q\subseteq Q_i\subseteq 48n \,Q.\]
    Indeed, the first inclusion is immediate. For the latter, if $y\in Q_i$ and $x\in 8\sqrt{n}Q\cap Q_i$, by \eqref{eq: lema: MPhi local a MPhi local diadica - eq1} we have
    \begin{align*}
    |x_Q-y|&\leq |x_Q-x|+|x-y|\\
    &\leq 8n\,\ell(Q)+\sqrt{n}\ell(Q_i)\\
    &< 8n\,\ell(Q)+16n\,\ell(Q)\\
    &= 24n\,\ell(Q),
    \end{align*}
    so $Q_i\subseteq B(x_Q, 24 n\ell(Q))\subseteq 48n\,Q.$
    
    By our choice of $Q_i$ we must have
	\[M_{\Phi, 8\sqrt{n}Q,\mathcal{D}^{(i)}}(f\mathcal{X}_Q)\leq M_{\Phi, Q_i,\mathcal{D}^{(i)}}(f\mathcal{X}_Q)=M_{\Phi,Q_i}^{\mathscr{D}}(f\mathcal{X}_Q).\]
	Finally,
	\[M_{\Phi, Q}f(x)\leq 3^n\sum_{i=1}^{3^n}M_{\Phi, Q_i}^{\mathscr{D}}(f\mathcal{X}_Q)(x).\qedhere\]
\end{proof}

We now proceed with the proof of the main result of this section.

\begin{proof}[Proof of Theorem~\ref{teo: tipo debil modular de MPhi,rho}]
Fixed $\theta\geq 0$, let $\sigma>0$ to be chosen later and observe that
\[M_\Phi^{\rho,\sigma}f(x)\leq \sup_{Q\in \mathcal{Q}_\rho} \|f\|_{\Phi,Q}+\sup_{Q\not\in\mathcal{Q}_\rho}\left(\frac{\rho(x_Q)}{r_Q}\right)^\sigma \|f\|_{\Phi,Q}=M_{\Phi,\rm{loc}}^{\rho}f(x)+M_{\Phi,\rm{glob}}^{\rho,\sigma}f(x),\]
recalling that $\mathcal{Q}_\rho$ is the set of subcritical cubes.
By Proposition~\ref{propo: cubrimiento critico}, there exists a sequence of critical cubes $Q_j=Q(x_j,\rho(x_j))$ that form a partition of $\mathbb{R}^n$ with controlled overlapping.  We can write
\begin{align*}
    w\left(\left\{x\in\mathbb{R}^n: M_{\Phi}^{\rho,\sigma} f(x)>\lambda\right\}\right)&\leq \sum_j w\left(\left\{x\in Q_j: M_{\Phi}^{\rho,\sigma} f(x)>\lambda\right\}\right)\\
    &\leq  \sum_j w\left(\left\{x\in Q_j: M_{\Phi, \rm{loc}}^{\rho} f(x)>\frac{\lambda}{2}\right\}\right)\\
    &\quad + \sum_j w\left(\left\{x\in Q_j: M_{\Phi, \rm{glob}}^{\rho,\sigma} f(x)>\frac{\lambda}{2}\right\}\right)\\
    &=I+II.
\end{align*}
Let us first estimate $I$. Fixed $Q_j$, there exists a cube $R_j=Q(x_j, C_\rho\, \rho(x_j))\supseteq Q_j$ such that if $Q\in\mathcal{Q_\rho}$ and $Q_j\cap Q\neq \emptyset$, then $Q\subseteq R_j$. Indeed, if $x\in Q_j\cap Q$ and $y\in Q$, by using \eqref{eq: constantesRho} we have
\begin{align*}
  |y-x_j|&\leq |y-x|+|x-x_j|\leq 2\rho(x_Q)+\rho(x_j)\leq 2^{N_0+1}C_0\rho(x)+\rho(x_j)\\
  &\leq 2^{N_0+2}C_0^2\rho(x_j)+\rho(x_j)
  \leq (2^{N_0+2}C_0^2+1)\rho(x_j),
\end{align*}
which implies that $y\in B(x_j, (2^{N_0+2}C_0^2+1)\rho(x_j))\subseteq Q(x_j, C_\rho \rho(x_j))$, with $C_\rho=\sqrt{n}(2^{N_0+2}C_0^2+1)$. Therefore, if $x\in Q_j$ by Lemma~\ref{lema: MPhi local a MPhi local diadica} we have 
\[M_{\Phi, \rm{loc}}^\rho f(x)\leq \sup_{Q\subseteq R_j}\|f\|_{\Phi, Q}=M_{\Phi, R_j}f(x)\leq 3^n\sum_{i=1}^{3^n} M_{\Phi,R_{i,j}}^{\mathscr{D}}f(x),\]
where $R_j\subseteq R_{i,j}\subseteq C_n R_j$, for every $i$. By virtue of Proposition~\ref{propo: tipo debil modular en un cubo} we obtain
\begin{align*}
  I&\leq \sum_j\sum_{i=1}^{3^n}w\left(\left\{x\in R_{i,j}: M_{\Phi,R_{i,j}}^{\mathscr{D}}f(x)>\frac{\lambda}{2\cdot 9^n} \right\}\right)\\
  &\leq C_{n,\rho,\theta}\sum_j \int_{R_{i,j}}\Phi\left(\frac{|f|}{\lambda}\right)M^{\rho,\theta}w \\
  & \lesssim \sum_j \int_{C_{n,\rho}Q_j}\Phi\left(\frac{|f|}{\lambda}\right)M^{\rho,\theta}w\\
  &\lesssim  \int_{\mathbb{R}^n}\Phi\left(\frac{|f|}{\lambda}\right)M^{\rho,\theta}w,
\end{align*}
where we have also used that $\Phi\in \Delta_2$ and item~\eqref{item: prop-cubrimientocritico - item b} of Proposition~\ref{propo: cubrimiento critico}.

We now estimate $II$. Fix $j$ and $x\in Q_j$. For $k\in\mathbb{N}$, we consider the sets
\[S_k=\{Q=Q(x_Q, r_Q)\ni x: b^{k-1}\rho(x_Q)<r_Q\leq b^k \rho(x_Q)\},\]
where $b=4^{N_0}$.
\begin{afirmacion}\label{af: teo: tipo debil modular de MPhi,rho - af1}
    There exist constants $C_1$ and $C_2$ greater than one such that  
    \[C_1^{-1}b^{-N_0k}\rho(x_j)\leq \rho(x_Q)\leq C_1b^{N_0k}\rho(x_j),\]
 and $Q\subseteq C_2d^kQ_j$, for every $Q\in S_k$ and $d=b^{N_0+1}$.
\end{afirmacion}

From this claim we get that $\|f\|_{\Phi,Q}\leq Cb^{2N_0kn}\|f\|_{\Phi, C_2d^k Q_j}$. Indeed, observe that
\begin{align*}
    \frac{1}{|Q|}\int_Q\Phi\left(\frac{|f|}{\|f\|_{\Phi, C_2d^k Q_j}}\right)&\leq \frac{|C_2d^k Q_j|}{|Q|}\frac{1}{|C_2d^kQ_j|}\int_{C_2d^kQ_j}\Phi\left(\frac{|f|}{\|f\|_{\Phi,C_2d^k Q_j}}\right)\\
    &\leq (C_2d^k)^n \left(\frac{r_j}{r_Q}\right)^n
    \leq (C_2d^k)^n \left(\frac{\rho(x_j)}{\rho(x_Q)}b^{1-k}\right)^n\\
    &\leq (C_2d^k)^n \left(C_1b^{N_0k}b^{1-k}\right)^n\\
    &=(C_1C_2b)^n b^{2N_0kn}.
\end{align*}
Therefore, we can write
\begin{align*}
    \sup_{Q\not\in \mathcal{Q}_\rho}\left(\frac{\rho(x_Q)}{r_Q}\right)^{\sigma}\|f\|_{\Phi,Q}&\leq \sup_{k\geq 1}\sup_{Q\in S_k}b^{(1-k)\sigma}\|f\|_{\Phi,Q}\\
    &\leq \sup_{k\geq 1} Cb^{2N_0kn}b^{(1-k)\sigma}\|f\|_{\Phi,C_2d^kQ_j}\\
    &\leq C b^{\sigma} \sum_{k=1}^{\infty} b^{2N_0kn}b^{-k\sigma} \|f\|_{\Phi,C_2d^kQ_j}.
\end{align*}
Notice that if we take $0<c<1/(N_0+1)$ and $\sigma>2N_0/(1-(N_0+1)c)$ we obtain that $b^{2N_0kn}b^{-k\sigma}\leq d^{-ck\sigma}$. This allows us to conclude that

\begin{equation}\label{eq: teo: tipo debil modular de MPhi,rho - eq1}
  M^{\rho,\sigma}_{\Phi,\textup{glob}}f(x)  \leq C \sum_{k\geq 1} d^{-kc\sigma}\|f\|_{\Phi, C_2d^kQ_j}=C \sum_{k\geq 1} d^{-kc\sigma}\|f\|_{\Phi, Q_j^k}
	 =A_j.  
\end{equation}
Let $J_1=\{j\in\mathbb{N}: A_j>\lambda/2\}$ and $J_2=\mathbb{N}\setminus J_1$. Since
 \[ w\left(\left\{x\in Q_j: M_{\Phi, \rm{glob}}^{\rho,\sigma} f(x)>\frac{\lambda}{2}\right\}\right)\leq  w\left(\left\{x\in Q_j: A_j>\frac{\lambda}{2}\right\}\right),\]
we have that $II\leq \sum_{j\in J_1}w(Q_j)$.

For every $j\in J_1$ we have that 
\[\frac{1}{|Q_j^k|}\int_{Q_j^k}\Phi\left(\frac{|f|}{C_0\lambda}\right)>1,\]
provided we pick  
\[C_0=\frac{d^{c\sigma}-1}{2C},\]
where $C$ is the constant appearing in \eqref{eq: teo: tipo debil modular de MPhi,rho - eq1}.
Indeed, if this inequality does not hold, we would get $\|f\|_{\Phi, Q_j^k}\leq C_0\lambda$, which yields
\[A_j\leq CC_0\lambda \sum_{k\geq 1} d^{-ck\sigma}\leq \frac{\lambda}{2},\]
which is a contradiction. Therefore, by virtue of \eqref{eq: relacion norma Luxemburgo con infimo} we can proceed as follows 
\begin{align*}
   II\leq \sum_{j\in J_1} w(Q_j)&\leq \frac{2C}{\lambda}\sum_{j\in J_1}\sum_{k\geq 1}d^{-kc\sigma}w(Q_j)\|f\|_{\Phi, Q_j^k}\\
     &\lesssim \frac{2C}{\lambda}\sum_{j\in J_1}w(Q_j)\sum_{k\geq 1}d^{-kc\sigma}\left(C_0\lambda+\frac{C_0\lambda}{|Q_j^k|}\int_{Q_j^k}\Phi\left(\frac{|f|}{C_0\lambda}\right)\right)\\
     &\leq 4CC_0\sum_{j\in J_1}\sum_{k\geq 1}d^{-kc\sigma}\frac{w(Q_j^k)}{|Q_j^k|}\int_{Q_j^k}\Phi\left(\frac{|f|}{C_0\lambda}\right)\\
     &\leq C\sum_{j\in J_1}\sum_{k\geq 1}d^{-kc\sigma}\left(1+\frac{r_{Q_j^k}}{\rho(x_j)}\right)^{\theta}\int_{Q_j^k}\Phi\left(\frac{|f|}{\lambda}\right)M^{\rho,\theta}w\\
     &\leq C\sum_{k\geq 1}d^{-kc\sigma}(1+C_2d^k)^\theta\int_{\mathbb{R}^n}\Phi\left(\frac{|f|}{\lambda}\right)\left(\sum_{j}\mathcal{X}_{Q_j^k}\right)M^{\rho,\theta}w.
\end{align*}

 By applying part~\eqref{item: prop-cubrimientocritico - item b} of Proposition~\ref{propo: cubrimiento critico}, we have that
\[II \lesssim \sum_{k\geq 1}d^{-kc\sigma+k\theta+kN_1}\int_{\mathbb{R}^n}\Phi\left(\frac{|f|}{\lambda}\right)M^{\rho,\theta}w.\]
Choosing $\displaystyle \sigma>\max\left\{\frac{\theta+N_1}{c}, \frac{2N_0}{1-(N_0+1)c}\right\}$ we conclude that 
\[II\leq C\int_{\mathbb{R}^n}\Phi\left(\frac{|f|}{\lambda}\right)M^{\rho,\theta}w,\]
and the proof is complete.
\end{proof}

We finish the section with the proof of Claim~\ref{af: teo: tipo debil modular de MPhi,rho - af1}.

\begin{proof}[Proof of Claim~\ref{af: teo: tipo debil modular de MPhi,rho - af1}]
    By means of \eqref{eq: constantesRho} we have that
    \begin{equation*}
      \rho(x_Q)\geq C_0^{-1}\rho(x_j)\left(1+\frac{|x_Q-x_j|}{\rho(x_j)}\right)^{-N_0}.  
    \end{equation*}
    On the other hand, since $Q\in S_k$ we can estimate
    \[|x_Q-x_j|\leq |x_Q-x|+|x-x_j|\leq r_Q+r_j\leq b^k\rho(x_Q)+\rho(x_j),\]
    and replacing on the inequality above we get
    \begin{equation}\label{eq: af: teo: tipo debil modular de MPhi,rho - af1 - eq1}
      \rho(x_Q)\geq C_0^{-1}\rho(x_j)\left(2+b^k\frac{\rho(x_Q)}{\rho(x_j)}\right)^{-N_0}.  
    \end{equation}
    If $2<b^k\rho(x_Q)/\rho(x_j)$, the estimate above leads us to
    \[\rho(x_j)\leq \left(2C_0^{1/N_0}b^k\right)^{N_0/(N_0+1)}\rho(x_Q)\leq 2C_0b^k\rho(x_Q).\]
    If we now assume that $b^k\rho(x_Q)/\rho(x_j)\leq 2$, we straightforwardly get from \eqref{eq: af: teo: tipo debil modular de MPhi,rho - af1 - eq1} that
    \[\rho(x_j)\leq 4^{N_0}C_0\rho(x_Q)\leq 2C_0b^k \rho(x_Q).\]
    On the other hand, again by \eqref{eq: constantesRho} and the fact that $Q\in S_k$ we have that
\begin{align*}
  \rho(x_Q)&\leq C_0\rho(x_j)\left(1+\frac{r_Q}{\rho(x_Q)}+\frac{r_j}{\rho(x_Q)}\right)^{N_0}\\
  &\leq C_0\rho(x_j)\left(2b^k+\frac{\rho(x_j)}{\rho(x_Q)}\right)^{N_0}.
\end{align*}
If $\rho(x_j)\leq 2b^k\rho(x_Q)$, we obtain $\rho(x_Q)\leq C_0(4b^k)^{N_0}\rho(x_j)$. If not, we arrive to 
\[\rho(x_Q)\leq 2^{N_0}C_0\rho(x_j)\left(\frac{\rho(x_j)}{\rho(x_Q)}\right)^{N_0},\]
 which after a rearrangement leads to $\rho(x_Q)\leq (2^{N_0}C_0)^{1/(N_0+1)}\rho(x_j)$. Therefore,
 \[\rho(x_Q)\leq C_0(4b^k)^{N_0}\rho(x_j).\]
 These estimates give us the first part of the claim by taking $C_1=4^{N_0}C_0$.
 
Now, for $z\in Q$ we have
\[|z-x_j|\leq |z-x|+|x-x_j|\leq 2r_Q+r_j\leq 2b^{k}\rho(x_Q)+\rho(x_j)\leq  (2C_1b^{k(N_0+1)}+1)\rho(x_j)\leq 4C_1b^{k(N_0+1)}\rho(x_j).\]
So, we obtain $z\in B(x_j, 4C_1b^{k(N_0+1)}\rho(x_j))\subseteq Q(x_j, 4\sqrt{n}C_1b^{k(N_0+1)}\rho(x_j))=4\sqrt{n}C_1b^{k(N_0+1)}Q_j$, so we can take $C_2=4\sqrt{n}C_1$.
\end{proof}

\section{Proof of Theorem~\ref{teo: equivalencias}}\label{section: prueba equivalencias}

We start this section by stating and proving some results that will be required for the proof of the main theorem. The first one is a consequence of Theorem~\ref{teo: tipo debil modular de MPhi,rho}.

\begin{lema}\label{lema: estimacion conjunto de nivel por integral}
Let $\varphi$ be a differentiable Young function in $\Delta_2$ and $w$ be a weight. For every $\theta\geq 0$, there exist positive constants $C$, $\sigma$ and $t_0>1$ such that the inequality
\[w(\{x\in\mathbb{R}^n: M_\varphi^{\rho,\sigma}f(x)>\lambda\})\leq C\int_{t_0}^\infty M^{\rho,\theta}w\left(\left\{x\in\mathbb{R}^n: 4t_0|f(x)|>\lambda s\right\}\right)\varphi'(s)\,ds\]
holds for every $\lambda>0$.
\end{lema}

\begin{proof}
    Since $\varphi$ is a Young function, there exists $t_0>1$ such that $\varphi(t_0)>0$. Fix $\lambda>0$ and $\theta\geq 0$. 
    We write $f=f_1+f_2$, where $f_1=f\mathcal{X}_{\{|f|\leq \lambda\}}$. Then, by Theorem~\ref{teo: tipo debil modular de MPhi,rho} there exist $C,\sigma>0$ such that
    \begin{align*}
        w\left(\left\{x\in\mathbb{R}^n: M_{\varphi}^{\rho,\sigma} f(x)>2\lambda\right\}\right)&\leq w\left(\left\{x\in\mathbb{R}^n: M_{\varphi}^{\rho,\sigma} f_2(x)>\lambda\right\}\right)\\
        &\leq C\int_{\{|f|>\lambda\}}\varphi\left(\frac{|f|}{\lambda}\right)M^{\rho,\theta}w.
    \end{align*}
Observe that
\begin{align*}
  \int_{\{|f|>\lambda\}}\varphi\left(\frac{|f|}{\lambda}\right)M^{\rho,\theta}w&\leq \int_0^\infty M^{\rho,\theta}w(\{x: |f(x)|/\lambda>\max\{1,s\}\})\varphi'(s)\,ds\\
  &=\int_0^{2t_0}+\int_{2t_0}^\infty\\
  &=I_1+I_2.
\end{align*}
For $I_1$ we have
\begin{align*}
  I_1\leq M^{\rho,\theta}w(\{x: |f(x)|>\lambda\})\int_0^{2t_0}\varphi'(s)\,ds&\leq C\int_{t_0}^{2t_0} M^{\rho,\theta}w(\{x: 2t_0|f(x)|>\lambda s\})\varphi'(s)\,ds\\
  &\leq C\int_{t_0}^\infty M^{\rho,\theta}w(\{x: 2t_0|f(x)|>\lambda s\})\varphi'(s)\,ds.
\end{align*}
On the other hand, since $2t_0>1$, it is clear that
\[I_2\leq \int_{2t_0}^\infty M^{\rho,\theta}w(\{x: |f(x)|>\lambda s\})\varphi'(s)\,ds\leq \int_{2t_0}^\infty M^{\rho,\theta}w(\{x: 2t_0|f(x)|>\lambda s\})\varphi'(s)\,ds.\]
Therefore we conclude that
\[ w\left(\left\{x\in\mathbb{R}^n: M_{\varphi}^{\rho,\sigma} f(x)>2\lambda\right\}\right)\leq C\int_{2t_0}^\infty  M^{\rho,\theta}w(\{x: 2t_0|f(x)|>\lambda s\})\varphi'(s)\,ds,\]
which yields the thesis by taking $\lambda$ instead of $2\lambda$.
\end{proof}

\medskip

Let us introduce the centred version of \eqref{eq: Maximal generalizada} over balls. Given a Young function $\varphi$ and $f\in L^{\varphi}_{\rm{loc}}$ we define
\[\mathcal{M}_{\varphi,c}^{\rho,\sigma} f(x)=\sup_{r>0}\left(1+\frac{r}{\rho(x)}\right)^{-\sigma}\|f\|_{\varphi, B(x,r)},\]
where $B(x,r)$ stands for the ball centred at $x$ with radius $r$. Particularly, when $\varphi(t)=t$ we write $\mathcal{M}_{\varphi,c}^{\rho,\sigma}=\mathcal{M}_{c}^{\rho,\sigma}$. It is easy to check that
\begin{equation}\label{eq: equivalencia entre centrada y no centrada}
    C_1 \mathcal{M}_{\varphi,c}^{\rho,\sigma} f(x)\leq \mathcal{M}_{\varphi}^{\rho,\sigma}f(x)\leq C_2 \mathcal{M}_{\varphi,c}^{\rho,\sigma/(N_0+1))} f(x),
\end{equation}
where $N_0$ is the constant appearing in \eqref{eq: constantesRho} and
\[\mathcal{M}_{\varphi}^{\rho,\sigma} f(x)=\sup_{B=B(x_B,r_B)\ni x}\left(1+\frac{r_B}{\rho(x_B)}\right)^{-\sigma}\|f\|_{\varphi, B}\]

(see, for example, \cite{BCHextrapolation}).

The following estimate of maximal operators applied to characteristic functions will be useful in the sequel.

\begin{lema}\label{lema: maximal de la caracteristica de una bola critica}
Let $\varphi$ be a Young function, $B_0=B(x_0, \rho(x_0))$ be a critical ball and $x\not\in 2B_0$. For every $\sigma>0$ we have that 
\begin{enumerate}[\rm(a)]
    \item\label{item: lema: maximal de la caracteristica de una bola critica - item a} $\displaystyle \mathcal{M}_c^{\rho,\sigma}(\mathcal{X}_{B_0})(x)\gtrsim \left(\frac{\rho(x_0)}{|x-x_0|}\right)^{n+\sigma(N_0+1)}$ and
    \item\label{item: lema: maximal de la caracteristica de una bola critica - item b} $\displaystyle \mathcal{M}^{\rho,\sigma}_{\varphi,c}(\mathcal{X}_{B_0})(x)\lesssim \left(\frac{\rho(x_0)}{|x-x_0|}\right)^{\sigma/(N_0+1)}\left[\varphi^{-1}\left(\left(\frac{|x-x_0|}{\rho(x_0)}\right)^n\right)\right]^{-1},$
\end{enumerate}
being $N_0$ the constant appearing in \eqref{eq: constantesRho}.
\end{lema}

\begin{proof}
    Let us start by proving \eqref{item: lema: maximal de la caracteristica de una bola critica - item a}. Fix any ball $B=B(x,r_B)$ and assume that $r_1\leq r_B\leq r_2$, where $r_1=|x-x_0|-\rho(x_0)$ and $r_2=|x-x_0|+\rho(x_0)$. Using equation \eqref{eq: constantesRho} we get that
    \begin{align*}
      1+\frac{r_B}{\rho(x)}\leq 1+\frac{|x-x_0|+\rho(x_0)}{\rho(x)}\leq 1+\frac{3|x-x_0|}{2\rho(x)}&\leq 1+\frac{3C_0|x-x_0|}{2\rho(x_0)}\left(1+\frac{|x-x_0|}{\rho(x_0)}\right)^{N_0}\\
      &\leq 3C_02^{N_0}\left(\frac{|x-x_0|}{\rho(x_0)}\right)^{N_0+1}.
    \end{align*}
 Therefore
    \begin{align*}
        \mathcal{M}_c^{\rho,\sigma}(\mathcal{X}_{B_0})(x)&\geq \sup_{r_1\leq r_B\leq r_2}\left(1+\frac{r_B}{\rho(x)}\right)^{-\sigma}\frac{|B\cap B_0|}{|B|}\\
        &\geq C\sup_{r_1\leq r_B\leq r_2}\left[3C_02^{N_0}\left(\frac{|x-x_0|}{\rho(x_0)}\right)^{N_0+1}\right]^{-\sigma}\frac{|B\cap B_0|}{(|x-x_0|+\rho(x_0))^n}\\
        &\geq C(3C_02^{N_0})^{-\sigma}\left(\frac{|x-x_0|}{\rho(x_0)}\right)^{-\sigma(N_0+1)}\left(\frac{\rho(x_0)}{|x-x_0|}\right)^n.
    \end{align*}
    In order to prove \eqref{item: lema: maximal de la caracteristica de una bola critica - item b}, we fix again $B=B(x, r_B)$ and assume that $r_B\geq |x-x_0|-\rho(x_0)$, since the average $\|\mathcal{X}_{B_0}\|_{\varphi, B}$ is zero otherwise.
    Then, for $\lambda>0$ we have that
    \[\frac{1}{|B|}\int_B \varphi\left(\frac{\mathcal{X}_{B_0}}{\lambda}\right)=\frac{|B\cap B_0|}{|B|}\varphi\left(\frac{1}{\lambda}\right)\leq 2^n\left(\frac{\rho(x_0)}{|x-x_0|}\right)^n\varphi\left(\frac{1}{\lambda}\right)\leq 1,\]
    provided we choose
    \[\lambda\geq \left[\varphi^{-1}\left(\left(\frac{|x-x_0|}{2\rho(x_0)}\right)^n\right)\right]^{-1}.\]
    On the other hand, again by \eqref{eq: constantesRho} we have 
    \begin{align*}
        1+\frac{r_B}{\rho(x)}\geq \frac{|x-x_0|-\rho(x_0)}{\rho(x)}
        \geq \frac{|x-x_0|}{2\rho(x)}&\geq \frac{|x-x_0|}{2}\left[C_0\rho(x_0)\left(1+\frac{|x_0-x|}{\rho(x_0)}\right)^{N_0/(N_0+1)}\right]^{-1}\\
        &\geq C_0^{-1}2^{-N_0/(N_0+1)}\left(\frac{|x_0-x|}{\rho(x_0)}\right)^{1-N_0/(N_0+1)},
    \end{align*}
    which implies that
    \[\left(1+\frac{r_B}{\rho(x)}\right)^{-\sigma} \leq (2C_0)^\sigma\left(\frac{|x_0-x|}{\rho(x_0)}\right)^{-\sigma/(N_0+1)}.\]
    Therefore 
    \begin{equation}\label{eq: lema: maximal de la caracteristica de una bola critica - eq1}
    \left(1+\frac{r_B}{\rho(x)}\right)^{-\sigma}\|\mathcal{X}_{B_0}\|_{\varphi, B}\leq (2C_0)^{\sigma} \left(\frac{|x-x_0|}{\rho(x_0)}\right)^{-\sigma/(N_0+1)}\left[\varphi^{-1}\left(\left(\frac{|x-x_0|}{\rho(x_0)}\right)^n\right)\right]^{-1},    
    \end{equation}
    for every $r_B\geq |x-x_0|+\rho(x_0)$. By taking supremum in $r_B$ we obtain \eqref{item: lema: maximal de la caracteristica de una bola critica - item b}. 
\end{proof}

We are now in a position to proceed with the main proof.

\begin{proof}[Proof of Theorem~\ref{teo: equivalencias}]
    Let us prove that \eqref{item: teo: equivalencias - item a} implies \eqref{item: teo: equivalencias - item b}. Fixed $\theta\geq 0$, by Lemma~\ref{lema: estimacion conjunto de nivel por integral}, there exist positive constants $C,\sigma$ and $t_0>1$ such that
    \[w(\{x\in\mathbb{R}^n: M_\eta^{\rho,\sigma}f(x)>\lambda\})\leq C\int_{t_0}^\infty M^{\rho,\theta}w\left(\left\{x\in\mathbb{R}^n: |f(x)|>\lambda s\right\}\right)\eta'(s)\,ds.\]
    Therefore,
    \begin{align*}
\int_{\mathbb{R}^n}\phi(M_\eta^{\rho,\sigma} f(x))w(x)\,dx&=\int_0^{\infty} a(\lambda)w(\{x\in\mathbb{R}^n: M_\eta^{\rho,\sigma}f(x)>\lambda\})\,d\lambda\\
        &\leq C\int_0^\infty a(\lambda)\int_{t_0}^\infty M^{\rho,\theta}w(\{x\in\mathbb{R}^n:|f(x)|>\lambda s\})\eta'(s)\,ds\,d\lambda\\
        &\leq C\int_0^\infty a(\lambda)\int_{1}^\infty M^{\rho,\theta}w(\{x\in\mathbb{R}^n:|f(x)|>\lambda s\})\eta'(s)\,ds\,d\lambda\\
        &=C\int_0^\infty M^{\rho,\theta}w(\{x: |f(x)|>s\})\left(\int_0^{s}\frac{a(\lambda)}{\lambda}\eta'(s/\lambda)\,d\lambda\right)\,ds\\
        &\leq C\int_0^\infty b(Cs) M^{\rho,\theta}w(\{x: |f(x)|>s\})\,ds\\
        &\leq C\int_{\mathbb{R}^n}\psi(C|f(x)|)M^{\rho,\theta}w(x)\,dx,
    \end{align*}
    from the definition of $\psi$ on \eqref{eq: definicion de phi y psi}.

\medskip
    
Observe that \eqref{item: teo: equivalencias - item c} follows from \eqref{item: teo: equivalencias - item b} by taking $f/(C_0\|f\|_{\psi, M^{\rho,\theta}w})$ instead of $f$, where $C_0>C^2$.

\medskip

We now prove that \eqref{item: teo: equivalencias - item c} implies \eqref{item: teo: equivalencias - item d}. Fix $f$ and take $\alpha=(\int_{\mathbb{R}^n}\psi(C|f(x)|)\,dx)^{-1}$. Clearly,
\[\int_{\mathbb{R}^n}\psi(C|f(x)|)M^{\rho,\theta}\alpha(x)\,dx=\int_{\mathbb{R}^n}\psi(C|f(x)|)\alpha\,dx=1.\]
Therefore, $\|f\|_{\psi,M^{\rho,\theta}\alpha}=1/C$. By hypothesis, we get that 
\[\|M_\eta^{\rho, \sigma} f \|_{\phi, \alpha}\leq C\|f\|_{\psi, M^{\rho,\theta}\alpha}=1,\]
which leads to
\[\int_{\mathbb{R}^n}\phi(M_\eta^{\rho,\sigma} f(x))\alpha\,dx \leq 1 =\int_{\mathbb{R}^n}\psi(C|f(x)|)\alpha\,dx,\]
which yields \eqref{item: teo: equivalencias - item d}.

\medskip

 Now we prove that \eqref{item: teo: equivalencias - item b} implies \eqref{item: teo: equivalencias - item e}. Fix $\theta\geq 0$, let $\sigma$ be the constant provided by \eqref{item: teo: equivalencias - item b} and let $\gamma\geq \sigma$. Fix $x$ and a cube $Q=Q(x_Q, r_Q)$ that contains $x$. From the generalized Hölder inequality we have that \refstepcounter{BP}\label{pag: estimacion de b implica e}
\begin{align*}
    \left(1+\frac{r_Q}{\rho(x_Q)}\right)^{-\gamma}\frac{1}{|Q|}\int_Q |f|u&\lesssim \left(1+\frac{r_Q}{\rho(x_Q)}\right)^{-\sigma}\|f\|_{\eta,Q}\left(1+\frac{r_Q}{\rho(x_Q)}\right)^{\sigma-\gamma}\|u\|_{\tilde\eta,Q}\\
    &\lesssim M_\eta^{\rho,\sigma} f(x)\,\,M_{\tilde\eta}^{\rho,\gamma-\sigma}u(x).
\end{align*}
By taking supremum over $Q$ we get that $M^{\rho,\gamma}(fu)(x)\lesssim M_\eta^{\rho,\sigma} f(x)\,\,M_{\tilde\eta}^{\rho,\gamma-\sigma}u(x)$.
Then, by \eqref{item: teo: equivalencias - item b} we have that 
\[\int_{\mathbb{R}^n}\phi\left(\frac{M^{\rho,\gamma}(fu)(x)}{M_{\tilde \eta}^{\rho,\gamma-\sigma}u(x)}\right)w(x)\,dx\leq \int_{\mathbb{R}^n} \phi\left(M_\eta^{\rho,\sigma}f(x)\right)w(x)\,dx\leq C\int_{\mathbb{R}^n}\psi(C|f(x)|)M^{\rho,\theta}w(x)\,dx.\]
Thus \eqref{item: teo: equivalencias - item e} can be obtained by replacing $f$ by $f/u$.

\medskip

We shall now prove that \eqref{item: teo: equivalencias - item e} implies \eqref{item: teo: equivalencias - item a}. By applying our hypothesis with $w$ and $\theta\geq 0$ to be chosen and $f$ replaced by $fu$ we get that
\[\int_{\mathbb{R}^n}\phi\left(\frac{M^{\rho,\sigma}(fu)(x)}{M_{\tilde\eta}\,u(x)}\right)w(x)\,dx\leq C\int_{\mathbb{R}^n}\psi(|f(x)|)M^{\rho,\theta}w(x)\,dx,\]
since we have chosen $\gamma=\sigma$. Fix $t>0$, $x_0\in\mathbb{R}^n$, consider the ball $B_0=B(x_0,\rho(x_0))$ and choose $f=t\mathcal{X}_{B_0}$. By taking $u=\mathcal{X}_{B_0}$ and combining Lemma~\ref{lema: maximal de la caracteristica de una bola critica} with \eqref{eq: equivalencia entre centrada y no centrada} we have that
\[M^{\rho,\sigma}(fu)(x)\gtrsim t\left(\frac{\rho(x_0)}{|x-x_0|}\right)^{n+\sigma(N_0+1)} \]
and
\[ M_{\tilde\eta}u(x)\lesssim \left[\tilde\eta^{-1}\left(\left(\frac{|x-x_0|}{\rho(x_0)}\right)^n\right)\right]^{-1}.\]
Therefore
\begin{align*}
\left(\int_{B_0}M^{\rho,\theta}w\right)\psi(t)&=\int_{\mathbb{R}^n}\psi(|f(x)|)M^{\rho,\theta}w(x)\,dx\\
&\geq C\int_{\mathbb{R}^n\setminus 2B_0}\phi\left(\frac{M^{\rho,\sigma}(fu)(x)}{M_{\tilde\eta}u(x)}\right)w(x)\,dx\\
&\gtrsim \int_{\mathbb{R}^n\setminus 2B_0}\phi\left(t\left(\frac{\rho(x_0)}{|x-x_0|}\right)^{n+\sigma(N_0+1)}\tilde\eta^{-1}\left(\frac{|x-x_0|}{\rho(x_0)}\right)^n\right)w(x)\,dx.
\end{align*}
Now we choose 
\[w(x)=\frac{\phi\left(t\left(\frac{\rho(x_0)}{|x-x_0|}\right)^{n}\tilde\eta^{-1}\left(\frac{|x-x_0|}{\rho(x_0)}\right)^n\right)}{\phi\left(t\left(\frac{\rho(x_0)}{|x-x_0|}\right)^{n+\sigma(N_0+1)}\tilde\eta^{-1}\left(\frac{|x-x_0|}{\rho(x_0)}\right)^n\right)}\mathcal{X}_{\mathbb{R}^n\setminus 2B_0}(x).\]
\begin{afirmacion}\label{af: finitud de maximal de w}
    There exist $\theta\geq 0$ and a positive constant $C$ that does not depend on $w$ such that
    \[\sup_{B_0} M^{\rho,\theta}w(x)\leq C.\]
\end{afirmacion}
We shall postpone the proof of this claim. If this estimate holds, then we can conclude that
\[C|B_0|\psi(t)\geq \int_{\mathbb{R}^n\setminus 2B_0}\phi\left(t\left(\frac{\rho(x_0)}{|x-x_0|}\right)^{n}\tilde\eta^{-1}\left(\left(\frac{|x-x_0|}{\rho(x_0)}\right)^n\right)\right)\,dx.\]
By means of \eqref{eq: relacion Phi - tilde Phi} we can continue the estimate as follows
\begin{align*}
    C|B_0|\psi(t)&\geq \int_0^\infty a(\lambda)\left|\left\{x\in\mathbb{R}^n\backslash 2B_0:\, t\left(\frac{\rho(x_0)}{|x-x_0|}\right)^{n}\tilde\eta^{-1}\left(\left(\frac{|x-x_0|}{\rho(x_0)}\right)^n\right)>\lambda\right\}\right|\,d\lambda\\
    &\geq \int_0^\infty a(\lambda)\left|\left\{x\in\mathbb{R}^n\backslash 2B_0:\, \frac{t}{2}>\eta^{-1}\left(\left(\frac{|x-x_0|}{\rho(x_0)}\right)^n\right)\lambda\right\}\right|\,d\lambda\\
    &= \int_0^\infty a(\lambda)\left|\left\{x\in\mathbb{R}^n\backslash 2B_0:\, \left[\eta\left(\frac{t}{2\lambda}\right)\right]^{1/n}\rho(x_0)>|x-x_0|\right\}\right|\,d\lambda.
\end{align*}
Notice that $\displaystyle 2^n<\frac{1}{2}\eta\left(\frac{t}{2\lambda}\right)$ if $0<\lambda<t/(2\eta^{-1}(2^{n+1}))$. Therefore,
\begin{align*}
    C|B_0|\psi(t)&\geq \int_0^{t/(2\eta^{-1}(2^{n+1}))} a(\lambda)\left|\left\{x\in\mathbb{R}^n\backslash 2B_0:\, \left[\eta\left(\frac{t}{2\lambda}\right)\right]^{1/n}\rho(x_0)>|x-x_0|\right\}\right|\,d\lambda\\
    &=C|B_0|\int_0^{t/(2\eta^{-1}(2^{n+1}))} a(\lambda)\eta\left(\frac{t}{2\lambda}\right)\,d\lambda\\
    &\geq C t|B_0|\int_0^{t/(2\eta^{-1}(2^{n+1}))}\frac{a(\lambda)}{\lambda}\eta'\left(\frac{t}{\lambda}\right)\,d\lambda,
\end{align*}
since $\eta$ is doubling and also it is easy to see that $\eta'(z)\approx \eta(z)/z$.
    Since $b$ is non-decreasing, we arrive to
    \[\int_0^{Ct} \frac{a(\lambda)}{\lambda}\eta'\left(\frac{t}{\lambda}\right)\,d\lambda\leq C\frac{\psi(t)}{t}\leq Cb(Ct),\]
    which gives us \eqref{item: teo: equivalencias - item a}

\medskip

We now prove that \eqref{item: teo: equivalencias - item d} implies \eqref{item: teo: equivalencias - item a}. Fixed $\theta\geq 0$ and any $0\leq \gamma\leq \sigma$, where $\sigma>0$ is provided by hypothesis, we proceed similarly as in page~\pageref{pag: estimacion de b implica e} with $w=1$ to obtain
\[\int_{\mathbb{R}^n}\phi\left(\frac{M^{\rho,\gamma}(fu)(x)}{M_{\tilde\eta}^{\rho,\sigma-\gamma}u(x)}\right)\,dx\leq C\int_{\mathbb{R}^n}\psi(|f(x)|)\,dx.\]
Now we can repeat the argument that we have used to prove that \eqref{item: teo: equivalencias - item e} implies \eqref{item: teo: equivalencias - item a} to get the desired result.  
\end{proof}

\begin{proof}[Proof of Claim~\ref{af: finitud de maximal de w}]
    Recall that
    \[w(x)=\frac{\phi\left(t\left(\frac{\rho(x_0)}{|x-x_0|}\right)^{n}\tilde\eta^{-1}\left(\frac{|x-x_0|}{\rho(x_0)}\right)^n\right)}{\phi\left(t\left(\frac{\rho(x_0)}{|x-x_0|}\right)^{n+\sigma(N_0+1)}\tilde\eta^{-1}\left(\frac{|x-x_0|}{\rho(x_0)}\right)^n\right)}\mathcal{X}_{\mathbb{R}^n\setminus 2B_0}(x).\]
    Fix $x\in B_0=B(x_0,\rho(x_0))$. By \eqref{eq: equivalencia entre centrada y no centrada} we have that
\[M^{\rho,\theta}w(x)\leq C \mathcal{M}_c^{\rho,\theta/(N_0+1)}w(x),\]
so it will be enough to estimate the centered version of $M^{\rho,\theta}$. Let $B=B(x,r)$ be any ball centered in~$x$. We shall prove that there exists $C>0$ such that
\[\left(1+\frac{r}{\rho(x)}\right)^{-\theta/(N_0+1)}\frac{1}{|B|}\int_B w(y)\,dy\leq C\]
for every $r>0$. Applying condition \eqref{eq: constantesRho} we have that
\[1+\frac{r}{\rho(x)}\geq 1+\frac{r}{C_0\rho(x_0)}\left(1+\frac{|x-x_0|}{\rho(x_0)}\right)^{-N_0/(N_0+1)}\geq \frac{1}{C_0}\left(1+\frac{|x-x_0|}{\rho(x_0)}\right)^{-N_0/(N_0+1)}\left(1+\frac{r}{\rho(x_0)}\right),\]
and consequently
\begin{align*}
  \left(1+\frac{r}{\rho(x)}\right)^{-\theta/(N_0+1)}&\leq C_0^{\theta/(N_0+1)}\left(1+\frac{|x-x_0|}{\rho(x_0)}\right)^{N_0\theta/(N_0+1)^2}\left(1+\frac{r}{\rho(x_0)}\right)^{-\theta/(N_0+1)}\\
  &\leq 2^{N_0\theta/(N_0+1)^2}C_0^{\theta/(N_0+1)}\left(1+\frac{r}{\rho(x_0)}\right)^{-\theta/(N_0+1)}\\
  &=C\left(1+\frac{r}{\rho(x_0)}\right)^{-\theta/(N_0+1)}.
\end{align*}
On the other hand, using \eqref{eq: relacion Phi - tilde Phi}, for $y\in B$ we can write
\begin{align*}
    w(y)&\leq \frac{\phi\left(2t\left[\eta^{-1}\left(\left(\frac{|y-x_0|}{\rho(x_0)}\right)^n\right)\right]^{-1}\right)}{\phi\left(t\left(\frac{\rho(x_0)}{|y-x_0|}\right)^{\sigma(N_0+1)}\left[2\eta^{-1}\left(\left(\frac{|y-x_0|}{\rho(x_0)}\right)^n\right)\right]^{-1}\right)}\mathcal{X}_{\mathbb{R}^n\setminus 2B_0}(y)\\
    &\leq \sum_{k=1}^\infty \frac{\phi\left(2t\left[\eta^{-1}\left(\left(\frac{|y-x_0|}{\rho(x_0)}\right)^n\right)\right]^{-1}\right)}{\phi\left(\frac{t}{2}\left(\frac{\rho(x_0)}{|y-x_0|}\right)^{\sigma(N_0+1)}\left[\eta^{-1}\left(\left(\frac{|y-x_0|}{\rho(x_0)}\right)^n\right)\right]^{-1}\right)}\mathcal{X}_{2^{k+1}B_0\setminus 2^kB_0}(y).
\end{align*}
From the doubling condition on $\phi$, there exists a positive constant $C_d$ such that
\[\phi(2z)\leq C_d\phi(z),\]
for every $z$.
If $y\in 2^{k+1}B_0\setminus 2^kB_0$, then
\[2^k\rho(x_0)<|y-x_0|\leq 2^{k+1}\rho(x_0).\]
Then we continue as follows
\begin{align*}
w(y)&\leq \sum_{k=1}^\infty \frac{\phi\left(2t\left[\eta^{-1}\left(\left(\frac{|y-x_0|}{\rho(x_0)}\right)^n\right)\right]^{-1}\right)}{\phi\left(\frac{1}{4}\,2^{(-k-1)\lceil\sigma\rceil(N_0+1)}2t\left[\eta^{-1}\left(\left(\frac{|y-x_0|}{\rho(x_0)}\right)^n\right)\right]^{-1}\right)}\mathcal{X}_{2^{k+1}B_0\setminus 2^kB_0}(y)\\
&\leq \sum_{k=1}^\infty C_d^{(k+1)\lceil\sigma\rceil(N_0+1)+2}\mathcal{X}_{2^{k+1}B_0\setminus 2^kB_0}(y)\\
&=C\sum_{k=1}^\infty C_d^{k\lceil \sigma \rceil(N_0+1)}\mathcal{X}_{2^{k+1}B_0\setminus 2^kB_0}(y).
\end{align*}
For $y\in 2^{k+1}B_0\backslash 2^k B_0$ we have
\[2^k<\frac{|y-x_0|}{\rho(x_0}\quad \text{ or equivalently }\quad k<\log_2\left(\frac{|y-x_0|}{\rho(x_0)}\right).\]
Also observe that $|y-x_0|\leq |y-x|+|x-x_0|<r+\rho(x_0)$. Therefore
\begin{align*}
    \left(1+\frac{r}{\rho(x)}\right)^{-\theta/(N_0+1)}\frac{1}{|B|}\int_B w&\leq C\sum_{k=1}^\infty 2^{-k}\left(1+\frac{r}{\rho(x_0)}\right)^{-\theta/(N_0+1)}2^{(\log_2 C_d)k\lceil \sigma \rceil(N_0+1)+k}\mathcal{X}_{2^{k+1}B_0\setminus 2^kB_0}\\
    &\leq C\sum_{k=1}^\infty 2^{-k}\left(1+\frac{r}{\rho(x_0)}\right)^{-\theta/(N_0+1)+(\log_2 C_d)\lceil \sigma \rceil (N_0+1)+1}\\ 
    &\leq C,
\end{align*}
provided we choose $\theta$ sufficiently large such that $-\theta/(N_0+1)+(\log_2 C_d)\lceil \sigma \rceil (N_0+1)+1<0$, where $\lceil \sigma \rceil$ is the greatest integer less or equal to $\sigma$.
\end{proof}

\medskip

Before proceeding to the proof of Theorem~\ref{teo: equivalencias ampliado}, we state a useful lemma which establishes a relation between $\|\cdot\|_\varphi$ and $\varrho_\varphi(\cdot)$. Although it holds true for a wider class of functions, we only need the case of Young functions (see Corollary 3.2.5 in \cite{HH19}).  

\begin{lema}\label{lema: relacion modular y norma}
    Let $\varphi$ be a Young function and $f\in L^\varphi$.
    \begin{enumerate}[\rm(a)]
        \item \label{item: lema: relacion modular y norma - item a} If $\|f\|_\varphi\leq 1$, then $\varrho_\varphi(f)\leq \|f\|_\varphi$.
        
        \item \label{item: lema: relacion modular y norma - item b} If $\|f\|_\varphi>1$, we have that $\|f\|_\varphi\leq \varrho_\varphi(f)$.
    \end{enumerate}
\end{lema}

\begin{proof}[Proof of Theorem~\ref{teo: equivalencias ampliado}]
The equivalence between items~\eqref{item: teo: equivalencias - item a} to~\eqref{item: teo: equivalencias - item e} is obtained from Theorem~\ref{teo: equivalencias}. Notice that \eqref{item: teo: equivalencias - item c} implies \eqref{item: teo: equivalencias - item f} by taking $w=1$. We shall prove that \eqref{item: teo: equivalencias - item f} implies  \eqref{item: teo: equivalencias - item d}. Let us first assume that $\|f\|_\Psi=2$. By hypothesis we get that 
\[\left\|M_\eta^{\rho,\sigma}\left(\frac{f}{2C}\right)\right\|_\phi\leq 1.\]
Applying item~\eqref{item: lema: relacion modular y norma - item a} of Lemma~\ref{lema: relacion modular y norma} we obtain
\[\int_{\mathbb{R}^n}\phi\left(M_\eta^{\rho,\sigma}\left(\frac{f}{2C}\right)\right)\leq \left\|M_\eta^{\rho,\sigma}\left(\frac{f}{2C}\right)\right\|_\phi\leq 1=\left\|\frac{f}{2}\right\|_\Psi\leq \int_{\mathbb{R}^n}\Psi\left(\frac{|f|}{2}\right),\]
by means of item~\eqref{item: lema: relacion modular y norma - item b} in Lemma~\ref{lema: relacion modular y norma}.

If $\|f\|_\Psi>0$, we define $\tilde f= 2f/\|f\|_\Psi$. Then $\|\tilde f\|_\Psi=2$ and using the case above we arrive to
\[\int_{\mathbb{R}^n}\phi\left(M_\eta^{\rho,\sigma}\left(\frac{\tilde f}{2C}\right)\right)\leq \int_{\mathbb{R}^n}\Psi\left(\frac{|\tilde f|}{2}\right).\]
If $g=\tilde f/2C$, the above inequality can be written as
\[\int_{\mathbb{R}^n}\phi\left(M_\eta^{\rho,\sigma}g\right)\leq \int_{\mathbb{R}^n}\Psi\left(C|g|\right),\]
which gives \eqref{item: teo: equivalencias - item d} and the proof is complete.
\end{proof}

\def\cprime{$'$}
\providecommand{\bysame}{\leavevmode\hbox to3em{\hrulefill}\thinspace}
\providecommand{\MR}{\relax\ifhmode\unskip\space\fi MR }
\providecommand{\MRhref}[2]{%
  \href{http://www.ams.org/mathscinet-getitem?mr=#1}{#2}
}
\providecommand{\href}[2]{#2}

\end{document}